\documentclass[11pt,reqno]{amsart}
\usepackage{esint}
\usepackage{bbm}
\usepackage{amssymb}
\usepackage{mathrsfs}
\usepackage{amsmath,amsthm}
\usepackage{url}
\usepackage{enumerate}
\usepackage{graphicx,xcolor}
\usepackage{times} 
\usepackage{setspace} 
\usepackage{titlesec} 
\usepackage[left=2.0cm, right=2.0cm, top=2.0cm, bottom=2.0cm]{geometry}
\usepackage{cite}
\usepackage[colorlinks=true,
citecolor=blue,
filecolor=red,
linkcolor=blue,
anchorcolor=red,
urlcolor=blue,
linktocpage=true,
pdfpagelabels=true,
bookmarksnumbered=true]{hyperref}

\numberwithin{equation}{section}

\newtheorem{thm}{Theorem}[section]
\newtheorem{cor}[thm]{Corollary}
\newtheorem{prop}[thm]{Proposition}

\newtheorem{lem}[thm]{Lemma}

\newtheorem{rem}{Remark}
\newtheorem{problem}{Problem}
\newtheorem{ex}[thm]{Example}

\newcommand{\be}{\begin{equation}}
	\newcommand{\ee}{\end{equation}}

\titleformat{\section}
{\fontsize{12}{14}\rmfamily\bfseries\large}
{\thesection}{1em}{}

\titleformat{\subsection}
{\fontsize{12}{14}\rmfamily\bfseries}
{\thesubsection}{1em}{}
\title{Pontrjagin Classes of Locally Spherically Symmetric Projectively Affine Sprays}
\author{Zuhuan Jia and Benling Li}
\date{
	{\small \today}
}
\thanks{
	Benling Li  is partly supported by the NNSFC(12471045),  NMNSF(2024J017) and K.C.
	Wong Magna Fund in Ningbo University. }

\begin{document}
	
		\maketitle
\begin{center}
\vskip -3mm
\small
School of Mathematics and Statistics, Ningbo University, 315211 Ningbo, China\\
{\it jzh\_mathematics@163.com (Z. Jia); libenling@nbu.edu.cn (B. Li) }
\end{center}
	\begin{abstract}
		By Chern--Weil theory, rational Pontrjagin classes are represented by curvature forms. Sprays, which require no metric and include all Finsler and Riemannian geodesic sprays, provide a natural setting for the vanishing problem. Shen and Zhao proved that if a manifold admits a locally projectively flat spray, then all of its rational Pontrjagin classes vanish. We ask whether this rigidity persists for the wider class of projectively affine sprays.
		
		We prove that if a smooth manifold admits a locally spherically symmetric
		projectively affine spray, then all rational Pontrjagin classes vanish.
		Stronger, the Pontrjagin forms of the associated Berwald spray vanish
		identically as differential forms, a genuine geometric rigidity. Hence a
		manifold with nonzero rational Pontrjagin class cannot carry such a spray; in particular, it admits no Riemannian metric whose geodesic spray is locally spherically symmetric.
		We also construct non-Berwald, non-projectively-flat examples with
		identically vanishing Pontrjagin forms, showing the phenomenon is not a
		byproduct of projective flatness.
	\end{abstract}


\noindent{\bf 2020 MR Subject Classification}: 53B40, 53C60.

	\section{\textbf{Introduction}}
	
	Pontrjagin classes are integral characteristic classes of real vector
	bundles and, by the Chern--Weil theory \cite{ref3}, admit
	explicit representatives as closed differential forms constructed from
	the curvature of any connection. A basic problem in differential
	geometry is to understand when these classes, or their Chern--Weil
	representatives, are forced to vanish by the existence of a special
	geometric structure. Pontrjagin classes and forms have also been studied in Riemannian and
pseudo-Riemannian geometry \cite{refDerdzinskiRoter, refLafontRoy, ref8,
refFerreiroMunoz,refdeKok}, and
we refer to \cite{ref9,ref10,ref12,ref7,ref13} for background on
characteristic classes.
	
	Spray geometry provides a natural framework in which geodesic equations
	are studied without reference to a metric. A spray on a smooth manifold
	$M$ is a vector field on the slit tangent bundle $TM\backslash\{0\}$,
	locally written as
	\[
	G = y^i\frac{\partial}{\partial x^i} - 2G^i(x,y)\frac{\partial}{\partial y^i},
	\]
	whose integral curves project to geodesics. Every Finsler metric, and in
	particular every Riemannian metric, induces a natural spray, but sprays
	form a strictly larger class of geometric objects.
The curvature tensors of a spray (the Riemann curvature, the Berwald
curvature, the Douglas curvature, and the \(\chi\)-curvature) have been
extensively studied in projective spray geometry and Finsler geometry;
see the survey \cite{LiShen2024Survey} and the references therein. The
Riemann and Berwald curvatures are treated in \cite{ref11, ref5}; the Douglas curvature in \cite{BucMuz, LiMoYu,
Yang,ref6}; and the
\(\chi\)-curvature in \cite{ref2, ref4}.

	A spray is called \emph{projectively affine} if its Douglas curvature
	vanishes. Geometrically, this means that its geodesics coincide, as
	point sets, with those of some affine connection; the Berwald connection
	of a general projectively affine spray is nonlinear, but it lies in the projective class of an affine connection. In 2024, Shen and Zhao
	\cite{ref1} proved the following striking result: if an $n$-dimensional
	manifold admits a \emph{locally projectively flat} spray, then all of
	its rational Pontrjagin classes vanish. Since local projective flatness
	is a projective condition, this theorem provides a genuine
	topological obstruction: a manifold with some nonzero rational
	Pontrjagin class cannot carry a locally projectively flat spray.

This raises the following natural question:
\begin{quote}
	\it Do analogous vanishing results hold for projectively affine
	sprays that are not necessarily locally projectively flat?
\end{quote}
As stated, the answer is negative. Every Riemannian metric induces a
geodesic Berwald spray, and hence projectively affine; however,
there exist Riemannian manifolds with nonzero rational Pontrjagin
classes (such as \(\mathbb{CP}^2\)). Thus, projective affineness alone
is too weak to force the vanishing of Pontrjagin classes, and an
additional symmetry condition is required.

The difficulty is also structural. For a general projectively affine
spray, the Berwald connection is nonlinear, so its curvature forms live
on the slit tangent bundle rather than on \(M\) itself. The Pontrjagin
forms computed from such a connection are not ordinary differential
forms on \(M\), and the algebraic simplifications available in the
locally projectively flat case of \cite{ref1} no longer apply.

We therefore impose the condition of \emph{local spherical symmetry}.
It is a classical condition with natural examples. The geodesic sprays
of Riemannian metrics with constant sectional curvature are globally spherically
symmetric. Another fundamental example is the Funk metric on the unit
ball,
\[
F(x,y)=\frac{\sqrt{|y|^2-(|x|^2|y|^2-\langle x,y\rangle^2)}+\langle x,y\rangle}{1-|x|^2},
\]
which is a spherically symmetric, projectively flat Finsler metric of constant
flag curvature \(K=-\tfrac14\) \cite{Shen2001}.
More generally, a Finsler
metric on an \(O(n)\)-invariant domain is spherically symmetric if and
only if \[ F=|y|\phi(|x|^2,\langle x,y\rangle/|y|),\] 
see \cite{ref19,ref20};
classification results and further examples can be found in
\cite{ref14,HuangMo,Li,Mo,ref16}. The geodesic sprays of all these
	examples are spherically symmetric in the sense of \eqref{GI}. Throughout this paper, \(|\cdot|\) and \(\langle \cdot,\cdot \rangle\)
denote the standard Euclidean norm and inner product on \(\mathbb{R}^n\),
respectively.

 In this paper, we work within the class of \emph{locally
spherically symmetric} sprays. On a Euclidean domain equipped with a fixed
Euclidean coordinate system, Gui and Li \cite{ref17} proved that a spray \(G\)
is invariant under the orthogonal group \(O(n)\) (i.e., \(G^i(Qx, Qy) = Q G^i(x,y)\)
for all \(Q \in O(n)\)) if and only if its coefficients take the form
\begin{equation}
	G^i = |y|A(r,s)y^i + |y|^2B(r,s)x^i, \label{GI}
\end{equation}
where \(A = A(r,s)\), \(B=B(r,s)\) are smooth functions and
\begin{equation}
	r=|x|^2, \quad s=\langle x,y\rangle/|y|. \label{rs}
\end{equation}
	 Motivated by this local normal form, we use spherical
	symmetry \emph{locally} on a general manifold: a spray on a smooth
	manifold $M$ is called locally spherically symmetric if $M$ admits an
	open cover $\{U_\alpha\}$ such that, in a coordinate system on each
	$U_\alpha$, the spray coefficients take the above form. The existence of a spherical chart is thus built into the definition,
rather than derived from \cite{ref17}. Besides local spherical symmetry,
another key notion is homogeneity, i.e., invariance under a transitive Lie
group action, whose framework was developed by Xu \cite{Xu2022}. While
homogeneity is a global condition, local spherical symmetry only requires
\(O(n)\)-invariant spray coefficients in local charts.

Our first main observation is a characterization of the projectively
affine condition. In Lemma \ref{g(r)=0} we prove that a spherically symmetric spray
is projectively affine if and only if
$
B(r,s) = f(r)s^2 + h(r)
$
for some smooth functions \(f(r),h(r)\). This is the key structural restriction
underlying the rest of the paper.

We are now in a position to state our main result.
\begin{thm}\label{The 1}
	Let \(M\) be an \(n\)-dimensional smooth manifold with \(n\ge2\). If
	\(M\) admits a locally spherically symmetric projectively affine
	spray, then all rational Pontrjagin classes of \(M\) vanish.
\end{thm}

The proof of Theorem \ref{The 1} proceeds as follows. We first show in
Proposition \ref{corD} that the projectively affine condition forces a
specific symmetry in the vertical derivatives of the Riemann curvature.
Using this symmetry, we pass to the associated spray of Shen \cite{ref2}
and obtain a Berwald spray whose curvature matrix splits into a symmetric
part and a part carrying a common factor. Applying Newton's identities to
this algebraic decomposition, we prove that all even elementary symmetric
functions of the curvature vanish identically, which yields the theorem.

The proof in fact yields a stronger geometric statement: the Pontrjagin
forms of the associated Berwald connection vanish identically as
differential forms, not merely as cohomology classes (see Remark
\ref{rem:stronger}). This pointwise vanishing is a genuine geometric
rigidity. Furthermore, no closedness assumption is required; the
argument is local and applies to arbitrary smooth manifolds.

	Let us compare Theorem \ref{The 1} with the existing literature. Shen and
Zhao \cite{ref1} proved the vanishing of rational Pontrjagin classes for
locally projectively flat sprays on arbitrary manifolds. The present
paper relaxes local projective flatness to projective affineness, at the
cost of imposing local spherical symmetry. Thus, in the spherically
symmetric setting, it strictly extends the main theorem of \cite{ref1}.

	The scope of the theorem is not restricted to projective flatness or to
	Berwald sprays. In Section \ref{ex} we prove the following existence
	result.
	
	\begin{thm}\label{thm-existence}
		There exist spherically symmetric projectively affine sprays on
		Euclidean domains that are non-Berwald and not of scalar curvature,
		and whose associated sprays have identically vanishing Pontrjagin
		forms. In particular, these sprays are non-projectively flat.
	\end{thm}
	
	Thus the vanishing phenomenon established in Theorem \ref{The 1} is a
	genuine rigidity of the projectively affine class under spherical
	symmetry, and not a consequence of more restrictive geometric
	hypotheses such as local projective flatness, Berwaldianity, or scalar
	curvature.
	
	Theorem \ref{The 1} has the following immediate consequence, which
	turns the vanishing theorem into a conditional topological obstruction.
	
	\begin{cor}\label{cor:obstruction}
		Let \(M\) be a smooth \(n\)-manifold, \(n\ge2\), and suppose that
	some rational Pontrjagin class of \(M\) is nonzero. Then \(M\) does
	not admit a locally spherically symmetric spray which is
	projectively affine.
	\end{cor}
	
This corollary provides a topological obstruction to the existence of
projectively affine sprays with local spherical symmetry. To see that
the obstruction is not vacuous, recall that the complex projective space
\(\mathbb{CP}^n\) (\(n \ge 2\)) has a nonzero first Pontrjagin class
(see, e.g., \cite{ref12}). More generally, the same holds for any
complex projective space \(\mathbb{CP}^n\) and for the quaternionic
projective space \(\mathbb{HP}^m\) with \(m \ge 2\). By
Corollary \ref{cor:obstruction}, neither of these manifolds admits a
locally spherically symmetric spray that is projectively affine.

Furthermore, the obstruction is not limited to closed manifolds.
The noncompact manifold \(\mathbb{CP}^2 \times \mathbb{R}\) also has a
nonzero rational Pontrjagin class, so the same conclusion applies.

The geodesic spray of a Riemannian metric is Berwald, hence projectively affine. If the metric is locally spherically symmetric in the sense that its geodesic spray takes the form \eqref{GI} in local charts, Theorem \ref{The 1} applies. This yields the following consequence for Riemannian metrics.

\begin{cor}\label{cor:riemann}
	Let \(M\) be an \(n\)-dimensional smooth manifold with \(n\ge2\). If
	\(M\) admits a Riemannian metric whose geodesic spray is locally
	spherically symmetric, then all
	rational Pontrjagin classes of \(M\) vanish. Equivalently, if some
	rational Pontrjagin class of \(M\) is nonzero, then \(M\) admits no
	such Riemannian metric.
\end{cor}

\begin{rem}\label{rem:riemann-local}
	In the literature, a \emph{spherically symmetric Riemannian manifold}
	is often defined globally: it requires a single chart covering the
	entire manifold with image a ball \(B_R\subset\mathbb{R}^n\), and a
	metric of the warped product form \(g=dr^2+\phi^2(r)\tilde{g}\) in
	polar coordinates. Such a manifold is diffeomorphic to a ball (or
	\(\mathbb{R}^n\)), hence contractible, and its Pontrjagin classes
	vanish trivially. The condition in Corollary \ref{cor:riemann} is
	strictly local and applies to the geodesic spray rather than the
	global metric. Consequently, the corollary covers Riemannian metrics
	on topologically non-trivial manifolds whose geodesic sprays are
	locally spherically symmetric, extending the classical vanishing
	phenomenon beyond the contractible setting.
\end{rem}

In the Riemannian setting, the space forms are the natural examples of
globally spherically symmetric sprays; their tangent bundles are
stably trivial, so all their rational Pontrjagin classes vanish and
they do not provide examples with nonzero Pontrjagin class.
	In the general
	(non-Riemannian) spray setting the round sphere provides a genuine
	globally spherically symmetric spray, but its tangent bundle again has
	vanishing Pontrjagin classes. This leads to the following natural question.
	
\begin{problem}\label{prob:rigidity}
Let \(M\) be a smooth manifold admitting a locally spherically
symmetric spray (not necessarily projectively affine). Must all
rational Pontrjagin classes of \(M\) vanish?
\end{problem}

A manifold with a locally spherically symmetric spray and some nonzero
rational Pontrjagin class would be of considerable interest. It would
show that local spherical symmetry alone does not force vanishing.
Moreover, by Theorem \ref{The 1}, such a spray cannot be projectively
affine; hence such a manifold would underscore the indispensable role of the
projectively affine hypothesis in Theorem \ref{The 1}. The existence
of such a manifold remains open.

	The paper is organized as follows. Section \ref{p2} reviews the basic
	notions of spray geometry and Chern--Weil theory. Section \ref{p3}
	establishes the geometric characterizations of spherically symmetric
	projectively affine sprays, including the key restriction
	$B(r,s)=f(r)s^2+h(r)$ (Lemma \ref{g(r)=0}), the symmetry
	$\Delta_{kl}=0$ (Proposition \ref{corD}), and the Berwald normal form
	(Theorem \ref{Thm 2}). Section \ref{p4} proves the vanishing theorem
	(Theorem \ref{The 1}) and records its Berwald and locally projectively
	flat consequences. Section \ref{ex} presents explicit examples,
	including non-Berwald and non-projectively-flat sprays, and shows the
	sharpness of the projectively affine condition (Theorem \ref{thm-existence}).

\section{Preliminaries}\label{p2}

In this section, we recall the basic notions of spray geometry, the
curvature tensors used below, the \(S\)-curvature, the \(\chi\)-curvature,
and the Chern--Weil representatives of Pontrjagin classes.
Standard
references are \cite{ref11} and the survey \cite{LiShen2024Survey} for
spray geometry, and \cite{ref12, ref10} for characteristic classes and
the Chern--Weil theory.
Throughout this paper, we use the Einstein summation convention:
repeated upper and lower indices are summed over their full range
(typically from \(1\) to \(n\)), unless otherwise explicitly stated.

A {\it spray} $G$ on an $n$-dimensional smooth manifold \(M\) is a vector
field on the slit tangent bundle $TM \setminus \{0\}$ which, in local
coordinates $(x^i,y^i)$, is expressed as
\begin{align}
	G= {y^i}\frac{\partial}{\partial x^i}-2{G^i}\frac{\partial}{\partial y^i},\label{G}
\end{align}
where the $G^i$ are positively homogeneous of degree two in $y$ and
satisfy
\begin{align}
	{G^i}(x,\lambda y)={\lambda}^2{G^i}(x,y), \quad\lambda > 0.\nonumber
\end{align}
The functions $G^i$ are called the {\it spray coefficients}.

The {\it geodesics} of $G$ are the projections of integral curves
$c(t)$ of $G$, i.e., $c=\pi\circ\tilde{c}(t)$. In this case,
$\tilde{c}(t)=\dot{c}(t)$. In local coordinates, the geodesics are
characterized by
\begin{align}
	\frac{\mathrm{d^2} x^i}{\mathrm{d} t^2}+2G^i \Bigl(x(t),\frac{\mathrm{d} x}{\mathrm{d} t} \Bigr)=0.\nonumber
\end{align}

Two sprays $\bar{G}$ and $G$ are called {\it projectively equivalent}
if there exists a function $P(x,y)$, positively homogeneous of degree
one in $y$, such that
\begin{align}
	\bar{G}^i(x,y)=G^i(x,y)+P(x,y)y^i.\nonumber
\end{align}
The scalar function $P(x,y)$ is called the {\it projective factor}.
Projectively equivalent sprays have the same geodesics as point sets,
up to reparametrization. A spray $G$ is called {\it locally
projectively flat} if every point admits a coordinate neighborhood in
which $G^i(x,y)=P(x,y)y^i$ for some positively homogeneous function $P$
of degree one; equivalently, its geodesics are locally straight lines.

Put
\begin{align}
	N^i_j:=\frac{\partial G^i}{\partial y^j},\ \quad \Gamma^i_{jk}:=\frac{\partial^2 G^i}{\partial y^j \partial y^k},\nonumber
\end{align}
the horizontal covariant derivative is given by
\begin{align}
	\frac{\delta}{\delta x^i}=\frac{\partial}{\partial x^i}-N^j_i\frac{\partial}{\partial y^j}.\nonumber
\end{align}
To simplify notation, we denote $A^*_{*|j}=\frac{\delta}{\delta x^j}A^*_*$ and $A^*_{*\cdot j}=\frac{\partial}{\partial y^j}A^*_*$.

The local connection $1$-forms and curvature forms of the Berwald
connection are given by
\begin{align}\label{Ber}
	\omega ^i_j=\Gamma ^i_{jk}dx^k,\ \quad \Omega^i_j:=d\omega^i_j-\omega^k_j \wedge \omega^i_k.
\end{align}
Let $\omega^i:=dx^i$ and $\omega^{n+i}:=dy^i+N^i_jdx^j$. Then the
connection forms $\omega^i_j$ satisfy
$
	d\omega^i=\omega^j \wedge \omega^i_j,\nonumber
$
and the curvature forms are given by
\begin{align}
	\Omega^i_j=\frac{1}{2}R_{j\;kl}^{\;i}dx^k\wedge dx^l-B_{jkl}^{i}dx^k\wedge \omega^{n+l},\label{o}
\end{align}
where $R_{j\;kl}^{\;i}$ and $B_{jkl}^{i}$ denote the Riemann curvature
tensor and the Berwald curvature tensor, respectively.

In local coordinates, the Riemann curvature and Berwald curvature are
given by
\begin{align}
	R_{j\;kl}^{\;i} &=\frac{\delta \Gamma^i_{jl}}{\delta x^k} - \frac{\delta \Gamma^i_{jk}}{\delta x^l} + \Gamma^i_{km} \Gamma^m_{jl} - \Gamma^i_{lm} \Gamma^m_{jk},\quad
	B_{jkl}^{i} = \frac{\partial^3 G^i}{\partial y^j \partial y^k \partial y^l}.\nonumber
\end{align}
The two-index {\it Riemann curvature tensor} of a spray is given by
$R^i_{\;k}=y^jR^{\;i}_{j\;kl}y^l$ and hence
\begin{align}
	R_{\;k}^{i}=2\frac{\partial G^i}{\partial x^k} - \frac{\partial^2 G^i}{\partial x^j\partial y^k}y^j+2{G^j}\frac{\partial^2 G^i}{\partial y^j\partial y^k}-\frac{\partial G^i}{\partial y^j}\frac{\partial G^j}{\partial y^k}.\nonumber
\end{align}
$R^i_{\;k}$ and  $R^{\;i}_{j\;kl}$ determine each other by the
following identity:
\begin{align}
	R^{\;i}_{j\;kl}=\frac{1}{3}\left(\frac{\partial^2 R^i_{\;k}}{\partial y^j \partial y^l}-\frac{\partial^2 R^i_{\;l}}{\partial y^j \partial y^k}\right).\label{5}
\end{align}

A spray with vanishing Berwald curvature \(B^i_{jkl}=0\) is called a
{\it Berwald spray}; equivalently, \(G^i=\tfrac12\Gamma^i_{jk}(x)y^jy^k\).
The {\it Douglas curvature tensor} $D^i_{jkl}$ \cite{Douglas} is
\begin{align}\label{Dou}
	D^i_{jkl}=B^i_{jkl}-\frac{2}{n+1}\left[ E_{jk}\delta^i_l+E_{jl}\delta^i_k+E_{kl}\delta^i_j+\frac{\partial E_{jk}}{\partial y^l}y^i \right],
\end{align}
where $E_{ij}=\tfrac12B^m_{ijm}$. A spray is {\it projectively affine}
if \(D^i_{jkl}=0\).
The Douglas curvature is projectively invariant: it is unchanged under
the projective deformation \(\bar{G}^i=G^i+Py^i\) \cite[Chapter~3]{ref11}.
Since \(B^i_{jkl} =0\Rightarrow E_{ij}=0\), Berwald sprays are
projectively affine; the converse fails, because projective affineness
only constrains the projective class of the geodesics, not the quadratic
dependence of \(G^i\) on \(y\).

We record the following hierarchy of projective conditions, which will
be used throughout the paper:
\[
\text{locally projectively flat}
\;\subsetneq\;
\text{projectively affine (Douglas $=0$)}
\;\subsetneq\;
\text{general sprays}.
\]
The first inclusion is strict: a projectively affine spray has
geodesics that are locally those of some affine connection, but these
geodesics need not be straight lines. The Douglas curvature is
projectively invariant: it is unchanged under the projective
deformation $\bar{G}^i=G^i+Py^i$. This invariance is what makes the
Douglas curvature the natural obstruction to projective flatness, and
it is the reason why the projectively affine condition plays a central
role in the present paper.

Let $R=R(x,y)$ and $\gamma_k=\gamma_k(x,y)$ be functions satisfying
$\gamma_ky^k=R$. A spray is said to be of {\it scalar curvature} if
\begin{align}
	R^i_{\;k}=R\delta^i_k-\gamma_{k}y^i.\nonumber
\end{align}
In particular, $G$ is said to be of {\it isotropic curvature} if it is
of scalar curvature with $\gamma_{k}=\frac{1}{2}R_{\cdot k}$, i.e.,
\begin{align}
	R^i_{\;k}=R\delta^i_k-\frac{1}{2}R_{\cdot k}y^i.\nonumber
\end{align}

Let $G$ be a spray of the form \eqref{G} on a manifold \(M\),
$dv=\sigma(x)|dx^1\cdots dx^n|$, $\sigma(x)>0$, be a smooth positive
density on \(M\). Then the {\it $S$-curvature} of $(G,dv)$ is given by
\begin{align}
	S=\frac{\partial G^m}{\partial y^m}-y^m\frac{\partial}{\partial x^m}\ln{ \sigma }.\nonumber
\end{align}

The {\it $\chi$-curvature} $\chi=\chi_k\mathrm{d}x^k$ is defined by
$
	\chi_k=\frac{1}{2}(S_{\cdot k|m}y^m-S_{|k}).\nonumber
$
In local coordinates,
\begin{align}
	\chi_k=\frac{1}{2}\left( y^i\frac{\partial^2 \Pi}{\partial x^i \partial y^k}-2G^i\frac{\partial^2 \Pi}{\partial y^i    \partial y^k}-\frac{\partial \Pi}{\partial x^k} \right),\label{chi}
\end{align}
where $\Pi : =\frac{\partial G^m}{\partial y^m}$ and $\chi_k$ is
independent of $dv$. The equation \eqref{chi} can be found in
\cite{ref21}.

Let $\varepsilon$ be a real vector bundle of rank $r$, and let
$\varepsilon _{\mathbb{C}}=\varepsilon \otimes_\mathbb{R} \mathbb{C}$ be
its complexification. Define the total Chern class of
$\varepsilon_{\mathbb{C}}$ by
\[
c(\varepsilon_{\mathbb{C}}):=1+c_1(\varepsilon_{\mathbb{C}})+c_2(\varepsilon_{\mathbb{C}})+\cdots +c_r(\varepsilon_{\mathbb{C}}),
\]
where $c_k(\varepsilon_{\mathbb{C}})\in H^{2k}(M;\mathbb{Z})$ and
\(H^{2k}(M;\mathbb{Z})\) denotes the \(2k\)-th singular cohomology group
of \(M\) with integer coefficients. Then the \(k\)-th Pontrjagin class of
the real vector bundle \(\varepsilon\) is given by
\begin{align}
	p_k(\varepsilon)=(-1)^kc_{2k}(\varepsilon_{\mathbb{C}})\in H^{4k}(M;\mathbb{Z}).\nonumber
\end{align}

Given a connection $D$ on $\varepsilon$, the Pontrjagin classes can also
be represented by differential forms representing de Rham cohomology
classes \cite{ref3}. Let $\Omega=(\Omega^i_j)$ as in \eqref{o} be the
curvature matrix of $D$. Define the forms $\sigma_i(\Omega)$ by
\begin{align}
	\det(I+t\Omega) = \sum_{r=0}^{n}\sigma_r(\Omega)t^r,\label{p}
\end{align}
where $\sigma_i(\Omega)$ is the $i$-th elementary symmetric polynomial
in the entries of $\Omega$, with multiplication interpreted as wedge
products. With the normalization adopted in \cite[p.~310, Corollary
C.10]{ref12}, the form $\sigma_{2k}(\Omega)\in\Omega^{4k}(M)$ is closed
and the image of $p_k(\varepsilon)$ in real cohomology satisfies
\[
p_k(\varepsilon) = \frac{1}{(2\pi)^{2k}} \bigl[\sigma_{2k}(\Omega)\bigr] \in H^{4k}(M;\mathbb{R}).
\]
Consequently, the total real Pontrjagin class
\[
p(\varepsilon) = 1+p_1(\varepsilon)+p_2(\varepsilon)+\cdots
\]
is determined by the even-degree coefficients of the determinant
$\det(I+t\Omega)$.

	\section{\textbf{Geometric characterizations of spherically symmetric
			projectively affine sprays}}\label{p3}
	
	\subsection{Curvature identities and the projectively affine
		condition}\label{p3.1}
	
	In this subsection, we establish the geometric characterizations of
	spherically symmetric projectively affine sprays that will be used in
	Section \ref{p4}. We first recall the coordinate conventions and the
	explicit \(\chi\)-curvature, then derive the skew-symmetric vertical
	derivative identities, and finally prove that the projectively affine
	condition forces \(B(r,s)=f(r)s^2+h(r)\), which in turn implies
	\(\Delta_{kl}=0\), where $\Delta_{kl}$ is defined in \eqref{DE}.
	
	Here and throughout, subscripts \(``r"\) and \(``s"\) denote partial
	derivatives with respect to \(r\) and \(s\), respectively. For example,
	\(A_s = \partial A /\partial s\), \(A_{rs} =  \partial^2 A / \partial r\partial s \),
	\(A_{ss}=  \partial^2 A / \partial s^2 \), etc.
	
	In local spherical coordinates, indices are raised and lowered using the
	Euclidean metric \(\delta_{ij}\); in particular,
	$
	y_m=\delta_{ml}y^l.
	$
	By \eqref{rs}, we have \(r=|x|^2, s= \langle x,y \rangle / |y|\).
	Consequently,
	\begin{align}
		r_{x^k}=2x_k,\quad r_{\cdot k}=0;\quad s_{x^k}=\frac{y_k}{|y|},\quad
		s_{\cdot k}=\frac{\partial s}{\partial y^k}=\frac{1}{|y|^2}(|y|x_k-sy_k).\label{s}
	\end{align}

We first recall the local decomposition of the Riemann curvature. As
shown in \cite{ref18}, the two-index Riemann curvature tensor of a
spherically symmetric spray takes the form
\begin{equation}
	R^i_{\;k} = R\delta^i_k - \iota_k x^i - \gamma_k y^i,\label{sym}
\end{equation}
with \(\gamma_ky^k=R\) and \(\iota_ky^k=0\).
	
	\begin{prop}\label{prop1}
		Let \(G\) be a spherically symmetric spray on an open subset
		\(\Omega\subset \mathbb{R}^n\) with \(G^i\) expressed in \eqref{GI}.
		Then the \(\chi\)-curvature of \(G\) is given by
		\begin{align}
			\chi_k&=-\frac{1}{2}(\vert y \vert x_k-sy_k)H(r,s),\label{chi_1}
		\end{align}
		where
		\begin{align}
			H(r,s)&=\nonumber(-s^2+r)\left[ (-2s^2B+2rB-1)B_{sss}-2sB_{rss}-6sBB_{ss}+6BB_s+2B_{rs}  \right] \\
			& \quad +(n+1)\left[ 2 \left[(r-s^2)A_{ss}- s A_s + A\right] B -2 s A_{rs} -A_{ss} + 2 A_r\right].\label{h0}
		\end{align}
	\end{prop}

\begin{proof}
	Let \(\Pi:=\frac{\partial G^m}{\partial y^m}\). Since \(s_{\cdot k}y^k=0\)
	by \eqref{rs} and \eqref{s}, a direct computation gives
	\[
	\Pi=\frac{\partial}{\partial y^m}\bigl(|y|A\,y^m+|y|^2B\,x^m\bigr)
	=|y|\Psi(r,s), \label{Psi}
	\]
where
		\begin{align}\label{Psi}
			\Psi (r,s)=(n+1)A+2sB+(r-s^2)B_s.
		\end{align}
	Differentiating \eqref{Psi} with respect to \(x^k\) and \(y^i\) yields
	\begin{align}\label{Pi}
		\frac{\partial \Pi}{\partial x^k}=2|y|\Psi_r x_k+\Psi_s y_k,\qquad
		\frac{\partial \Pi}{\partial y^i}=\Psi_s x_i+\frac{1}{|y|}(\Psi-s\Psi_s)y_i.
	\end{align}
	Using \(x_iy^i=s|y|\), \(y_iy^i=|y|^2\) and \(\delta_{ik}y^i=y_k\), and
	substituting \(s_{\cdot k}\) from \eqref{s}, we obtain
	\begin{align}\label{Piyi}
		y^i\frac{\partial^2\Pi}{\partial x^i\partial y^k}
		=(2s\Psi_r+\Psi_s)y_k+(2s\Psi_{rs}+\Psi_{ss})(|y|x_k-sy_k).
	\end{align}
	Differentiating the second identity in \eqref{Pi} with respect to \(y^k\)
	yields
	\begin{align}\label{Piy}
		\frac{\partial^2\Pi}{\partial y^i\partial y^k}
		=s_{\cdot k}x_i\Psi_{ss}-\frac{1}{|y|}s_{\cdot k}y_is\Psi_{ss}
		+\frac{1}{|y|^3}(\Psi-s\Psi_s)(|y|^2\delta_{ik}-y_iy_k).
	\end{align}
	Contracting \eqref{Piy} with \(x^i\) and \(y^i\), and substituting
	\eqref{s}, gives
	\begin{align}
		x^i\frac{\partial^2\Pi}{\partial y^i\partial y^k}
		&=\frac{1}{|y|^2}\bigl[(r-s^2)\Psi_{ss}+\Psi-s\Psi_s\bigr](|y|x_k-sy_k),
		\quad
		y^i\frac{\partial^2\Pi}{\partial y^i\partial y^k}=0,
		\notag
	\end{align}
	and hence
	\begin{align}\label{2G}
		2G^i\frac{\partial^2\Pi}{\partial y^i\partial y^k}
		=2B\bigl[(r-s^2)\Psi_{ss}+\Psi-s\Psi_s\bigr](|y|x_k-sy_k).
	\end{align}
	Substituting \eqref{Pi}, \eqref{Piyi}, and \eqref{2G} into \eqref{chi}
	yields
	\[
	\chi_k=-\frac12\bigl[2s\Psi_{rs}+\Psi_{ss}-2\Psi_r
	-2B\bigl((r-s^2)\Psi_{ss}+\Psi-s\Psi_s\bigr)\bigr](|y|x_k-sy_k),
	\]
	i.e., \(\chi_k=-\frac12(|y|x_k-sy_k)H(r,s)\) with
	\begin{align}\label{Hr}
		H(r,s)=2\Psi_r-2s\Psi_{rs}-\Psi_{ss}
		+2B\bigl[(r-s^2)\Psi_{ss}+\Psi-s\Psi_s\bigr].
	\end{align}
Finally, differentiating \eqref{Psi} yields
		\begin{align}
			\Psi_r&\nonumber=(n+1)A_r+2sB_r+B_s+(r-s^2)B_{rs}, \quad
			\Psi_s =(n+1)A_s+2B+(r-s^2)B_{ss},\\ \nonumber
			\Psi_{rs}&=(n+1)A_{rs}+2B_r+B_{ss}+(r-s^2)B_{rss},\quad
			\Psi_{ss} =(n+1)A_{ss}+2B_s-2sB_{ss}+(r-s^2)B_{sss}.
		\end{align}
Substituting these expressions into \eqref{Hr}, we  obtain \eqref{h0}.  This completes the
	proof.
\end{proof}

The Riemann curvature of a spherically symmetric spray was first computed
in \cite{ref18}:
\begin{align}
	R_{\;k}^{i} &= Y(r,s) \left(|y|^2 \delta_{k}^{i} - y^k y^i \right)+P(r,s)\left(|y| x^k-s y^k \right) |y| x^i +X(r,s) \left( |y| x^k - s y^k \right) y^i,\label{R}
\end{align}
where
\begin{align}
	P(r,s) & = \left[2(r-s^2)B-1\right]B_{ss} -2sB_{rs}  -(r-s^2) B_s^2-2sBB_s+4B_r+4B^2,\label{H1} \\
X(r,s) & =2 \left[(r-s^2)A_{ss}- s A_s + A\right] B -2 s A_{rs} -A_{ss} + 4 A_r - A A_s - \left[(r-s^2) A_s +s A+1 \right] B_s, \notag \\
Y(r,s) &= \left[ 2(r-s^2)B-1 \right] A_s- 2 s A_r + 2(1 + sA)B + A^2.  \notag
\end{align}

The following identities will be used in the proof of Theorem \ref{The 1}.
\begin{lem}\label{lem-skew}
	Let \(G\) be a spherically symmetric spray as above. Then
	\begin{align}
		\chi_{k\cdot l}-\chi_{l\cdot k}&=\frac{1}{\vert y \vert}(x_ly_k-x_ky_l)H(r,s).\label{chic}\\
		\iota_{k\cdot l}-\iota_{l\cdot k}&=3(x_ly_k-x_ky_l)P(r,s),\label{iotac}\\
		\gamma_{k\cdot l}-\gamma_{l\cdot k}&=\frac{1}{\vert y \vert}(x_ly_k-x_ky_l)Q(r,s),\label{gammac}
	\end{align}
	where \(H(r,s)\) and \(P(r,s)\) are given in \eqref{h0} and \eqref{H1}, and
	\begin{align}
		Q(r,s)=2X(r,s)+Y_s(r,s)=6\left[(r-s^2)A_{ss}- s A_s + A\right] B -6s A_{rs} -3A_{ss} + 6 A_r.\label{h2}
	\end{align}
\end{lem}
\begin{proof}
	Differentiating \eqref{chi_1} with respect to \(y^l\), we obtain
	\[
	\chi_{k\cdot l}=-\frac12\Bigl[\Bigl(\frac{x_ky_l-x_ly_k}{|y|}+\frac{sy_ky_l}{|y|^2}-s\delta_{kl}\Bigr)H(r,s)+\frac{(|y|x_k-sy_k)(|y|x_l-sy_l)}{|y|^2}H_s(r,s)\Bigr],
	\]
	and antisymmetrizing in \(k,l\) yields \eqref{chic}. From \eqref{R},
	\[
	\iota_k=-|y|(|y|x_k-sy_k)P(r,s),\qquad
	\gamma_k=-(|y|x_k-sy_k)X(r,s)+y_kY(r,s).
	\]
	Differentiating with respect to \(y^l\) gives
	\begin{align}
		\iota_{k\cdot l}&=-\left(2x_ky_l-x_ly_k-|y|s\delta_{kl}\right)P(r,s)-\frac{(|y|x_k-sy_k)(|y|x_l-sy_l)}{|y|}P_s(r,s),\notag\\
		\gamma_{k\cdot l}&=\frac{1}{|y|^2}\Bigl[-|y|(x_ky_l-x_ly_k)X-sy_ky_lX+|y|^2(sX+Y)\delta_{kl}\notag\\
		&\qquad-(|y|x_k-sy_k)(|y|x_l-sy_l)X_s+y_k(|y|x_l-sy_l)Y_s\Bigr].\notag
	\end{align}
	A direct subtraction yields \eqref{iotac} and \eqref{gammac}.
\end{proof}

	We now establish a characterization of spherically symmetric projectively
	affine sprays. The following lemma shows that the vanishing of the
	Douglas curvature forces the coefficient \(B(r,s)\) to be quadratic in
	\(s\).

	\begin{lem}\label{g(r)=0}
		Let \(G\) be a spherically symmetric spray on an open subset
		\(\Omega\subset \mathbb{R}^n\), \(n \ge 2\), with spray coefficients
		given by \eqref{GI}. Then \(G\) is projectively affine if and only if
		\[
		B(r,s) = f(r)s^2 + h(r)
		\]
		for some smooth functions \(f(r)\) and \(h(r)\). Consequently,
		\[
		G^i = |y| A(r,s) y^i + \big[f(r)\langle x,y\rangle^2 + h(r)|y|^2\big] x^i.
		\]
	\end{lem}

\begin{proof}
	The Douglas curvature is projectively invariant. Since
	\(G^i=G_0^i+Py^i\) with \(G_0^i=|y|^2B(r,s)x^i\) and
	\(P=|y|A(r,s)\), the two sprays are projectively equivalent, hence
	\(D(G)=D(G_0)\). It therefore suffices to compute the Douglas curvature for \(G_0\),
	i.e., for \(A=0\).

	For \(G_0\), we have \(\Pi_0:=\frac{\partial G_0^m}{\partial y^m} =|y|\bar\Psi\) with
	\(\bar\Psi= 2sB+(r-s^2)B_s \), and
	\(B^i_{jkl}=x^i(|y|^2B)_{jkl}\). Substituting into \eqref{Dou}
	yields
	\begin{align}\label{D1}
		D^i_{jkl}= x^i(|y|^2 B)_{\cdot j\cdot k\cdot l}-\frac{1}{n+1}\Big[
		\big(|y|\bar{\Psi}\big)_{\cdot j\cdot k}\delta_l^i +\big(|y|\bar{\Psi}\big)_{\cdot j\cdot l}\delta_k^i +\big(|y|\bar{\Psi}\big)_{\cdot k\cdot l}\delta_j^i +\big(|y|\bar{\Psi}\big)_{\cdot j\cdot k\cdot l}y^i	\Big].
	\end{align}
	By \eqref{s}, a direct computation gives
	\begin{align}\label{yB}
		(|y|^2 B)_{\cdot j\cdot k\cdot l}
		&=|y|^2 B_{sss} s_{\cdot j}s_{\cdot k}s_{\cdot l}+ (B_s-sB_{ss})\Big[
		s_{\cdot j}\big(\delta_{kl}-\tfrac{1}{|y|^2}y_k y_l\big)\notag\\
		&\qquad + s_{\cdot k}\big(\delta_{jl}-\tfrac{1}{|y|^2}y_j y_l\big)+ s_{\cdot l}\big(\delta_{jk}-\tfrac{1}{|y|^2}y_j y_k\big)\Big].
	\end{align}

	Assume \(D^i_{jkl}=0\). Fix \(x\neq0\), \(y\neq0\) with \(r-s^2>0\), and set
	\(z=x-\frac{\langle x,y\rangle}{|y|^2}y\), so that
	\(\langle y,z\rangle=0\), \(\langle x,z\rangle=|z|^2=r-s^2\),
	\(s_{\cdot j}z^j=\frac{r-s^2}{|y|}\).
	Contracting \eqref{D1} with \(y_iz^jz^kz^l\) and substituting
	\eqref{yB} and \(\bar\Psi\) gives
	\[
	y_i D^i_{jkl} z^j z^kz^l=\frac{(r-s^2)^2}{n+1}\left[-(r-s^2)^2 B_{ssss}+(n+8)s(r-s^2)B_{sss}+3(n+3)s(B_s-sB_{ss})\right].
	\]
	Hence
	\begin{align}\label{B01}
		(r-s^2)^2 B_{ssss}-(n+8)s(r-s^2)B_{sss}-3(n+3)s(B_s-sB_{ss})=0.
	\end{align}
	Set \(F=(r-s^2)B_{sss}+3(B_s-sB_{ss})\). Since
	\(F_s=(r-s^2)B_{ssss}-5sB_{sss}\), equation \eqref{B01} becomes
	\[ (r-s^2)F_s-(n+3)sF=0, \]
 whose general solution is
	\(F=c(r)(r-s^2)^{-\frac{n+3}{2}}\). Smoothness of \(B\) at
	\(s=\pm\sqrt r\) forces \(c(r)=0\), hence \(F=0\), i.e.,
	\begin{align}\label{B2}
		(r-s^2)B_{sss}+3(B_s-sB_{ss})=0.
	\end{align}
	Using \((B_s-sB_{ss})_s=-sB_{sss}\), the equation \eqref{B2} becomes
	\[ (r-s^2)(B_s-sB_{ss})_s-3s(B_s-sB_{ss})=0, \]
so
	\(B_s-sB_{ss}=d(r)(r-s^2)^{-3/2}\); smoothness again forces
	\(d(r)=0\), hence \(B_s-sB_{ss}=0\). Equation \eqref{B2} then reduces
	to \((r-s^2)B_{sss}=0\), so \(B_{sss}=0\) on \(r-s^2>0\), and
	integrating three times in \(s\) yields
\begin{equation} B(r,s)=f(r)s^2+h(r). \label{BinLemma3.3} \end{equation}
Since \(r-s^2>0\) corresponds to \(x\) and \(y\) being non-parallel, the
argument shows that \eqref{BinLemma3.3} holds on the dense open subset
\(\{|s|<\sqrt r\}\) of the spherical chart. Since \(B\) is smooth, it
extends continuously to the boundary \(\{|s|=\sqrt r\}\); at the center
\(x=0\) (\(r=s=0\)), \(B(0,0)=h(0)\) by continuity. Hence
\eqref{BinLemma3.3} holds on the whole spherical chart.

Conversely, if \(B(r,s)=f(r)s^2+h(r) \), then
\(B_s-sB_{ss}=B_{sss}=B_{ssss}=0\), and \eqref{D1} gives
\(D_{jkl}^{i}(G_0)=0\). Since \(D(G)=D(G_0)\) by projective invariance,
\(G\) is projectively affine.
\end{proof}

By Lemma \ref{g(r)=0}, a projectively affine spray satisfies
\(B(r,s)=f(r)s^2+h(r)\). Substituting this into the curvature identities,
we find that a specific combination of the vertical derivatives of
\(\gamma_k\) and \(\chi_k\) vanishes. We record this symmetry as follows.
Define the antisymmetric tensor
\begin{equation}\label{DE}
	\Delta_{kl}:=\frac{1}{3}(\gamma_{k\cdot l}-\gamma_{l\cdot k})
	-\frac{1}{n+1}(\chi_{k\cdot l}-\chi_{l\cdot k}).
\end{equation}

	\begin{prop}\label{corD}
		Let \(G\) be a spherically symmetric projectively affine spray on a
		domain \(U\subset\mathbb{R}^n\), \(n\ge2\), with spray coefficients
		given by \eqref{GI}. Then \(\Delta_{kl}=0\); equivalently,
		\(\eta_{k\cdot l}=\eta_{l\cdot k}\), where
		\[
		\eta_k:=\frac13\gamma_k-\frac{1}{n+1}\chi_k.
		\]
	\end{prop}

	\begin{proof}
		By Lemma \ref{g(r)=0}, \(B(r,s)=f(r)s^2+h(r)\).
		Substituting \eqref{chic} and \eqref{gammac} into \eqref{DE} yields
		\begin{align}\label{De}
			\Delta_{kl}
			=\frac{1}{3|y|}(x_ly_k-x_ky_l)Q(r,s)
			-\frac{1}{(n+1)|y|}(x_ly_k-x_ky_l)H(r,s),
		\end{align}
		and substituting \eqref{h0} and \eqref{h2} into \eqref{De} gives
		\begin{align}\label{DD}
			\Delta_{kl}
			=\frac{-(r-s^2)(x_l y_k-x_k y_l)}{(n+1)|y|}
			\Big[(2rB-2s^2B-1)B_{sss}-2sB_{rss}
			-6sBB_{ss}+2B_{rs}+6BB_s\Big].
		\end{align}
		Now \(B(r,s)=f(r)s^2+h(r)\) implies
		\[
		B_s=2fs,\qquad B_{rs}=2f's,\qquad B_{ss}=2f,\qquad B_{rss}=2{f}',\qquad B_{sss}=0,
		\]
		and hence the bracket in \eqref{DD} becomes
		\[
		-4f's-6sB\cdot 2f+4f's+6B\cdot 2fs=-12sfB+12sfB=0.
		\]
	Hence \(\Delta_{kl}=0\). Since
	\(\Delta_{kl}=\eta_{k\cdot l}-\eta_{l\cdot k}\) with
	\(\eta_k=\frac13\gamma_k-\frac{1}{n+1}\chi_k\), this is exactly \(\eta_{k\cdot l}=\eta_{l\cdot k}\).
	\end{proof}

	\subsection{Berwald normal forms}\label{p3.2}
	
Within the class of spherically symmetric projectively affine sprays,
Lemma \ref{g(r)=0} already fixes \(B(r,s)=f(r)s^2+h(r)\), so the
Berwald condition constrains only \(A(r,s)\). The following theorem
gives the corresponding normal form.

	\begin{thm}\label{Thm 2}
		Let \(G\) be a spherically symmetric projectively affine spray on an open subset \(\Omega\subset\mathbb{R}^n\) ($n\geq 2$) with spray coefficients
		given by \eqref{GI}. Then \(G\) is a Berwald spray if and only if
		there exist smooth functions \(a(r),f(r),h(r)\)  such that
		\begin{align}
			A(r,s)=a(r)s,\qquad B(r,s)=f(r)s^2+h(r),\nonumber
		\end{align}
		equivalently,
	$
		G^i = a(r)\langle x,y\rangle y^i + \left(f(r)\langle x,y\rangle^2+h(r)|y|^2 \right)x^i.
	$
	\end{thm}

\begin{proof}
	By Lemma \ref{g(r)=0},
	\(G^i=|y|A(r,s)y^i+[f(r)\langle x,y\rangle^2+h(r)|y|^2]x^i\), and the second
	term is already quadratic in \(y\). Hence \(G\) is a Berwald spray if and only
	if \(P(x,y)y^i\) is quadratic in \(y\) for every \(i\), where
	\(P:=|y|A(r,s)\).

	Suppose \(G\) is a Berwald spray. Since \(P(x,y)y^i\) is homogeneous
	quadratic in \(y\) for each \(i\) and \(n\ge2\), \(P\) is homogeneous
	linear, i.e., \(P(x,y)=c_j(x)y^j\). Spherical symmetry
	\(P(Qx,Qy)=P(x,y)\) forces \(c(x)\) to be invariant under the
	stabilizer of \(x\) in \(O(n)\); hence
	\(c_j(x)=a(r)x_j\). Therefore \(P=a(r)\langle x,y\rangle\), and
	\(A=a(r)s\).

	Conversely, if \(A=a(r)s\), then
	\(G^i=a(r)\langle x,y\rangle y^i+[f(r)\langle x,y\rangle^2+h(r)|y|^2]x^i\)
	is quadratic in \(y\), so \(G\) is Berwald.
\end{proof}

	\section{\textbf{Vanishing of rational Pontrjagin classes}}\label{p4}
	
	\subsection{The associated spray and the algebraic vanishing
		lemma}\label{p4.1}

In this subsection we introduce the associated spray and the algebraic
tools for the proof of Theorem~\ref{The 1}; the geometric input
\(\Delta_{kl}=0\) from Proposition~\ref{corD} will be used in
Section~\ref{p4.2}.

Let \(G\) be a spray and \(dv\) a volume form on an \(n\)-manifold \(M\).
The spray \emph{associated with} \((G,dv)\) is
\begin{align}
	\hat{G}^i=G^i-\frac{S}{n+1}y^i, \label{1}
\end{align}
where \(S\) is the \(S\)-curvature of \((G,dv)\); we say that \(\hat{G}\)
and \(G\) are \emph{projectively equivalent}. The following result is
due to Shen \cite{ref2}.

\begin{lem}[\cite{ref2}]\label{lem 1}
	If \(\hat{G}\) and \(G\) are projectively equivalent with \eqref{1},
	then
	\[
	\hat{B}_{jkl}^{i}=D_{jkl}^{i},
	\]
	and the Riemann curvature of \(\hat{G}\) is
	\begin{align}
		\hat{R}^i_{\;k}=R^i_{\;k}+\tau \delta^i_k-\frac{1}{2}\tau_{\cdot k}y^i+\frac{3\chi_k}{n+1}y^i,
		\qquad
		\tau:=\Bigl(\frac{S}{n+1}\Bigr)^2+\frac{1}{n+1}S_{|m}y^m.\label{2}
	\end{align}
\end{lem}

\begin{rem}\label{rem-dv}
	The spray \(\hat G\) depends on \(dv\) through \(S\). The proof of
	Theorem~\ref{The 1} uses only two facts: \(\hat G\) is Berwald
	whenever \(G\) is projectively affine (Lemma~\ref{lem 1}), so
	\(\hat\Omega\) is the curvature of a linear connection on \(M\); and
	in a local spherical chart the terms involving \(\tau\) enter the
	curvature decomposition only through the second vertical derivative
	\(\tau_{\cdot j\cdot l}\), which is symmetric. Hence the decomposition
	\(\hat\Omega=\mathcal S+\mathcal T\) below, and the vanishing of
	\(\sigma_{2k}(\hat\Omega)\), do not depend on the choice of \(dv\).
\end{rem}

Assume \(D_{jkl}^{i}=0\). By Lemma \ref{lem 1}, \(\hat{B}_{jkl}^{i}=0\),
so the Berwald connection of \(\hat{G}\) is a torsion-free linear
connection on \(M\), with
$
\hat{G}^i=\tfrac12\hat{\Gamma}^{i}_{jk}(x)y^jy^k,\,
\hat{\Gamma}^{i}_{jk}=\hat{\Gamma}^{i}_{kj}.
$
Its curvature forms are ordinary differential forms on \(M\):
\begin{align}
	\hat{\Omega}^i_{j}=\frac{1}{2}\hat{R}_{j\;kl}^{\;i}(x)dx^k\wedge dx^l.\label{3}
\end{align}
By Chern--Weil theory \cite{ref3},
\begin{align}
	p_k=\frac{1}{(2\pi)^{2k}}[\sigma_{2k}(\hat{\Omega})],\label{4}
\end{align}
where \(\sigma_{2k}(\hat{\Omega})\) is the \(2k\)-th elementary
symmetric polynomial of the curvature matrix \(\hat{\Omega}\), i.e.,
	\begin{align}
		\sigma_{2k}(\hat{\Omega})&=\nonumber \underset{i_1<i_2<\cdots<i_{2k},\sigma \in S_{2k}}{\Sigma}sgn(\sigma)\hat{\Omega}^{i_{\sigma(1)}}_{i_1}\wedge \cdots \wedge \hat{\Omega}^{i_{\sigma(2k)}}_{i_{2k}}\\
		&=\frac{1}{2^{2k}}\underset{i_1<i_2<\cdots<i_{2k},\sigma \in S_{2k}}{\Sigma}sgn(\sigma)\hat{R}_{i_1m_1l_1}^{i_{\sigma(1)}} \cdots \hat{R}_{i_{2k}m_{2k}l_{2k}}^{i_{\sigma(2k)}}dx^{m_1}\wedge dx^{l_1}\cdots dx^{m_{2k}}\wedge dx^{l_{2k}},\nonumber
	\end{align}
	where \(S_{2k}\) is the symmetric group on \(2k\) elements.

Since \(\sigma_{2k}\) will be used repeatedly, we regard it below as an
algebraic operation on matrix-valued \(2\)-forms: for any \(n\times n\)
matrix \(\tilde{\Omega}=(\tilde{\Omega}^i_j)\) of \(2\)-forms, not
necessarily a curvature matrix, \(\sigma_{2k}(\tilde{\Omega})\) denotes
the \(2k\)-th elementary symmetric polynomial with wedge products.

The Pontrjagin classes of Douglas sprays were computed in \cite{ref1}.

\begin{lem}[\cite{ref1}]
	For a Douglas spray \(G\) on \(M\), the Pontrjagin classes \(p_k\)
	of \(M\) take the form
	\[
	p_k=\frac{1}{(4\pi)^{2k}(2k)!}\sum_{i_1,\dots,i_{2k},\sigma\in S_{2k}}
	\operatorname{sgn}(\sigma)\prod_{s=1}^{2k}\hat{R}^{\;i_{\sigma(s)}}_{i_s\; m_s l_s}
	\bigwedge_{s=1}^{2k}dx^{m_s}\wedge dx^{l_s},
	\]
	where \(\hat{R}\) is the Riemann curvature of the associated spray
	\(\hat{G}\).
\end{lem}

For a matrix \(\mathcal{K}=(\mathcal{K}^i{}_j)\) of even-degree
differential forms, define
\[
\det(I+t\mathcal{K})=\sum_{k=0}^{n}\sigma_k(\mathcal{K})t^k,
\]
where \(\sigma_k(\mathcal{K})\) is the \(k\)-th elementary symmetric
expression of \(\mathcal{K}\), with wedge products. We shall use the
classical Newton identities; see Macdonald \cite[Chapter~I, \S2,
p.~23, Eq.~(2.11')]{Macdonald}.

\begin{lem}[\cite{Macdonald}]\label{lem-newton}
Let \(\mathcal{A}\) be a commutative \(\mathbb{R}\)-algebra and let
	\(\mathcal{K}\in\operatorname{Mat}_n(\mathcal{A})\) (the algebra of
	\(n\times n\) matrices over \(\mathcal{A}\)).
Write
	\(\det(I+t\mathcal{K})=\sum_{k=0}^{n}\sigma_k(\mathcal{K})t^k\)
	with \(\sigma_0(\mathcal{K})=1\). Then for \(1\leq k\leq n\),
	\[
	k\,\sigma_k(\mathcal{K})=\sum_{i=1}^{k}(-1)^{i-1}\sigma_{k-i}(\mathcal{K})\operatorname{tr}(\mathcal{K}^i).
	\]
\end{lem}

We apply the Newton identities to a special class of matrix-valued
\(2\)-forms; the resulting algebraic lemma is the core of the proof of
Theorem \ref{The 1}.

\begin{lem}\label{lem2}
	Let \(\mathcal S=(\mathcal S^i_j)\) and \(\mathcal T=(\mathcal T^i_j)\)
	be matrices of \(2\)-forms with
	\[
	\mathcal S^i_j= A_{jl}\,dx^i\wedge dx^l,
	\qquad A_{jl}=A_{lj},
	\]
	and
	\[
	\mathcal T^i_j = x^i\left(\tilde{H} \,dx^j+\tilde{H} _{\cdot j}\theta\right)\wedge\omega,
	\]
	where \(\tilde{H}=\tilde{H}(r,s)\) is smooth,
	\(\tilde{H}_{\cdot j}:=\partial\tilde{H}/\partial y^j\), and
	\begin{align}
		\omega= x_l\,dx^l,\quad \theta= y_l\,dx^l.\label{ot}
	\end{align}
	Set \(\mathcal K=\mathcal S+\mathcal T\). Then \(\sigma_k(\mathcal K)=0\)
	for all \(k\ge 2\);
in particular, for all \(m\ge 1\),
		\begin{align}
			\sigma_{2m}(\mathcal K)=0.\label{sig}
		\end{align}
\end{lem}

\begin{proof}
Under the Einstein summation convention, the \((i,k)\)-th entry of
	\(\mathcal S^2\) is
	\[
	(\mathcal S^2)^i_k = \mathcal S^i_j \wedge \mathcal S^j_k
	= (A_{jl}\,dx^i\wedge dx^l) \wedge (A_{kp}\,dx^j\wedge dx^p)
	= A_{jl}A_{kp} \, dx^i\wedge dx^l\wedge dx^j\wedge dx^p.
	\]
	Since \(A_{jl}=A_{lj}\) is symmetric in the indices \(j,l\), while the
	\(2\)-form \(dx^l\wedge dx^j = -dx^j\wedge dx^l\) is skew-symmetric,
	their contraction vanishes identically. Therefore, \(\mathcal S^2=0\).

Next, since every entry of \(\mathcal T\) contains the common factor \(\omega\),
	we have \(\mathcal T^i_j\wedge \mathcal T^p_q=0\) for all indices,
	which implies \(\mathcal T^2=0\).

	For $\mathcal T\mathcal S$, substituting the expressions of $\mathcal T$
	and $\mathcal S$ gives
	\[
	(\mathcal T\mathcal S)^i_k = \mathcal T^i_j\wedge \mathcal S^j_k
	=  x^iA_{kp}\,(\tilde H\,dx^j+\tilde H_{\cdot j}\theta)
	\wedge\omega\wedge dx^j\wedge dx^p.
	\]
	The $\tilde H$-term vanishes since $dx^j\wedge\omega\wedge dx^j=0$. For
	the $\tilde H_{\cdot j}$-term, use
	$\tilde H_{\cdot j}=\frac{\tilde H_s}{|y|^2}(|y|x_j-sy_j)$ and
	$ (|y|x_j-sy_j)\,dx^j=|y|\omega-s\theta$: the term equals
	\[
	  \frac{\tilde H_s}{|y|^2}x^iA_{kp}
	\bigl(|y|\omega-s\theta\bigr)\wedge\theta\wedge\omega\wedge dx^p=0,
	\]
	because $\omega\wedge\theta\wedge\omega=0$ and $\theta\wedge\theta=0$.
	Hence $\mathcal T\mathcal S=0$.

Thus $\mathcal S^2= \mathcal T^2= \mathcal T\mathcal S=0,$ and hence
		\[
		{\mathcal K}^2=(\mathcal S+\mathcal T)^2=\mathcal S^2+\mathcal S\mathcal T+\mathcal T\mathcal S+\mathcal T^2=\mathcal S\mathcal T.
		\]
Since
	$\operatorname{tr}(\mathcal S)=  A_{il}dx^i\wedge dx^l=0$
	and $\operatorname{tr}(\mathcal S\mathcal T)=\operatorname{tr}(\mathcal T\mathcal S)$
	by the cyclic property for matrix-valued $2$-forms, we get
	$\operatorname{tr}(\mathcal K^2)=0$. Also
	$\mathcal K^3=\mathcal S(\mathcal T\mathcal S)+\mathcal S\mathcal T^2=0$,
so for all integer $k\ge2$,
		\[
		\operatorname{tr}({\mathcal{K}}^k)=0.
		\]

	Now $\operatorname{tr}(\mathcal K)=\operatorname{tr}(\mathcal S) +\operatorname{tr}(\mathcal T) = \operatorname{tr}(\mathcal T)
	=\sum_ix^i\tilde H_{\cdot i}\theta\wedge\omega$, hence
	$\operatorname{tr}(\mathcal K)\wedge\operatorname{tr}(\mathcal K)=0$.
	Apply Newton's identity (Lemma~\ref{lem-newton}) in $\operatorname{Mat}_n(\Omega^{\mathrm{even}}(N))$:
\[
		2\sigma_2(\mathcal{K})=\operatorname{tr}(\mathcal{K})\wedge\operatorname{tr}(\mathcal{K})-\operatorname{tr}(\mathcal{K}^2)=0.
		\]
		So \(\sigma_2(\mathcal{K})=0\). Suppose inductively that \vskip -5mm
		\[
		\sigma_i(\mathcal{K})=0,
		\qquad
		2\leq i\leq k-1.
		\]
Then Newton's identity gives
	$k\,\sigma_k(\mathcal K)= \sum_{i=1}^{k}(-1)^{i-1}\sigma_{k-i}(\mathcal{K})\operatorname{tr}(\mathcal{K}^i) = \sigma_{k-1}(\mathcal K)\operatorname{tr}(\mathcal K)=0$.
	Thus $\sigma_k(\mathcal K)=0$ for all $k\ge2$; in particular,
	$\sigma_{2m}(\mathcal K)=0$ for $m\ge1$.
\end{proof}

	\begin{rem}\label{rem-sign-T}
		The overall sign of \(\mathcal T\) is irrelevant for the conclusion of
		Lemma \ref{lem2}.  In the application below, combining \eqref{ST} with the subsequent calculations, we obtain \eqref{T1}, i.e.,
		\[
		\mathcal T_j^i=-x^i\big(P(r,s)\,dx^j+P_{\cdot j}(r,s)\,\theta\big)\wedge\omega,
		\]
		so that Lemma \ref{lem2} applies with \(\tilde H=-P\).
	\end{rem}

\subsection{Proof of the vanishing theorem}\label{p4.2}
	
We now prove Theorem \ref{The 1}: by Proposition \ref{corD} the
projectively affine condition gives \(\eta_{k\cdot l}=\eta_{l\cdot k}\),
which makes the curvature matrix of the associated spray decompose as
required by Lemma \ref{lem2}.

\begin{proof}[Proof of Theorem \ref{The 1}]
	Let \(G\) be a locally spherically symmetric projectively affine spray
	on \(M\). Fix a local spherical chart \(U\subset M\); all quantities
	below are restricted to \(U\) and we omit the dependence on \(U\).

	By \eqref{sym}, the Riemann curvature has the form
	$R^i_{\;k}=R\delta^i_k-\iota_kx^i-\gamma_ky^i$, where $\gamma_ky^k=R$, $\iota_ky^k=0$.
	Since $R^{\;i}_{j\;kl}=\frac13(R^i_{\;k\cdot l\cdot j}-R^i_{\;l\cdot k\cdot j})$
	by \eqref{5}, we have
	\begin{align}
		R^{\;i}_{j\;kl}
		&=\frac13\Bigl[(\iota_{l\cdot k\cdot j}-\iota_{k\cdot l\cdot j})x^i
		+(\gamma_{l\cdot k\cdot j}-\gamma_{k\cdot l\cdot j})y^i
		+(\gamma_{l\cdot j}\delta^i_k-\gamma_{k\cdot j}\delta^i_l) \notag\\
		&\qquad\quad+(\gamma_{l\cdot k}-\gamma_{k\cdot l})\delta^i_j
		+(R_{\cdot l\cdot j}\delta^i_k-R_{\cdot k\cdot j}\delta^i_l)\Bigr].\label{Rjk}
	\end{align}
A straightforward calculation of equation (\ref{2}) yields
	\begin{align}\label{hatR}
		\hat{R}_{j\;kl}^{\;i}
		&=R_{j\;kl}^{\;i}+\frac12(\tau_{\cdot j\cdot l}\delta^i_k-\tau_{\cdot j\cdot k}\delta^i_l)
		+\frac1{n+1}\Bigl[(\chi_{k\cdot l\cdot j}-\chi_{l\cdot k\cdot j})y^i \notag\\
		&\qquad+(\chi_{k\cdot j}\delta^i_l-\chi_{l\cdot j}\delta^i_k)
		+(\chi_{k\cdot l}-\chi_{l\cdot k})\delta^i_j\Bigr],
	\end{align}
	and collecting the $y^i$- and $\delta^i_j$-terms of \eqref{Rjk} and
	\eqref{hatR} yields
	\begin{align}
		\hat{R}_{j\;kl}^{\;i}
		&=\frac13\bigl(R_{\cdot l\cdot j}\delta^i_k-R_{\cdot k\cdot j}\delta^i_l\bigr)
		+\frac12\bigl(\tau_{\cdot j\cdot l}\delta^i_k-\tau_{\cdot j\cdot k}\delta^i_l\bigr)
		+\frac13(\iota_{l\cdot k\cdot j}-\iota_{k\cdot l\cdot j})x^i \notag\\
		&\qquad+\Bigl[\frac13(\gamma_{l\cdot k\cdot j}-\gamma_{k\cdot l\cdot j})
		+\frac1{n+1}(\chi_{k\cdot j\cdot l}-\chi_{l\cdot j\cdot k})\Bigr]y^i +\Bigl(\frac13\gamma_{l\cdot j}-\frac1{n+1}\chi_{l\cdot j}\Bigr)\delta^i_k \notag\\
		&\qquad-\Bigl(\frac13\gamma_{k\cdot j}-\frac1{n+1}\chi_{k\cdot j}\Bigr)\delta^i_l - \Bigl[\frac13(\gamma_{k\cdot l}- \gamma_{l\cdot k})
		+\frac1{n+1}(\chi_{k\cdot l}-\chi_{l\cdot k})\Bigr]\delta^i_j.
		 \notag
	\end{align}
	By Proposition \ref{corD}, $\eta_k=\frac13\gamma_k-\frac1{n+1}\chi_k$ and
	$\Delta_{kl}=\frac13(\gamma_{k\cdot l}-\gamma_{l\cdot k})
	-\frac1{n+1}(\chi_{k\cdot l}-\chi_{l\cdot k})=0$; hence the
	$y^i$- and $\delta^i_j$-terms vanish, and
	\begin{align}
		\hat{R}_{j\;kl}^{\;i}
		&=\frac13\bigl(R_{\cdot l\cdot j}\delta^i_k-R_{\cdot k\cdot j}\delta^i_l\bigr)
		+\frac12\bigl(\tau_{\cdot j\cdot l}\delta^i_k-\tau_{\cdot j\cdot k}\delta^i_l\bigr)
		+\frac13(\iota_{l\cdot k\cdot j}-\iota_{k\cdot l\cdot j})x^i
		+\eta_{l\cdot j}\delta^i_k-\eta_{k\cdot j}\delta^i_l.\label{hatR2}
	\end{align}

	Substituting \eqref{hatR2} into
	$\hat{\Omega}^i_j=\frac12\hat{R}_{j\;kl}^{\;i}dx^k\wedge dx^l$ and
	using the elementary identity
	$(R_{\cdot l\cdot j}\delta^i_k-R_{\cdot k\cdot j}\delta^i_l)dx^k\wedge dx^l
	=2R_{\cdot l\cdot j}dx^i\wedge dx^l$ (and the same identity with $\tau$
	and $\eta$ in place of $R$), we obtain
	\begin{align}\label{ST}
		\hat{\Omega}^i_j
		&=\underbrace{\Bigl(\frac13R_{\cdot l\cdot j}+\frac12\tau_{\cdot l\cdot j}+\eta_{l\cdot j}\Bigr)dx^i\wedge dx^l}_{=:\,\mathcal S^i_j}
		+\underbrace{\frac16(\iota_{l\cdot k}-\iota_{k\cdot l})_{\cdot j}x^i dx^k\wedge dx^l}_{=:\,\mathcal T^i_j}.
	\end{align}
	Since $R_{\cdot l\cdot j}=R_{\cdot j\cdot l}$,
	$\tau_{\cdot l\cdot j}=\tau_{\cdot j\cdot l}$ and
	$\eta_{l\cdot j}=\eta_{j\cdot l}$ (equivalent to $\Delta_{kl}=0$), the
	coefficient $A_{lj}=\frac13R_{\cdot l\cdot j}+\frac12\tau_{\cdot l\cdot j}+\eta_{l\cdot j}$
	is symmetric, so $\mathcal S$ satisfies the hypothesis of Lemma \ref{lem2}.

	For $\mathcal T$, use \eqref{iotac}:
	\[
	(\iota_{l\cdot k}-\iota_{k\cdot l})_{\cdot j}x^i
	=-3(x_l\delta_{kj}-x_k\delta_{lj})x^iP-3(x_ly_k-x_ky_l)x^iP_{\cdot j},
	\]
	where $P_{\cdot j}=s_{\cdot j} P_s $. Substituting into \eqref{ST} and
	using
	$(x_l\delta_{kj}-x_k\delta_{lj})dx^k\wedge dx^l=2dx^j\wedge x_ldx^l$,
	$(x_ly_k-x_ky_l)dx^k\wedge dx^l=2y_kdx^k\wedge x_ldx^l$,
	we get
	\begin{align}\label{T1}
		\mathcal T^i_j
		&=-x^iP\,dx^j\wedge\omega-x^iP_{\cdot j}\,\theta\wedge\omega
		=-x^i\bigl(P\,dx^j+P_{\cdot j}\theta\bigr)\wedge\omega,
	\end{align}
	where $\omega=x_ldx^l$, $\theta=y_ldx^l$. Thus $\mathcal T$ has the
	form required by Lemma \ref{lem2} with $\tilde H=-P$; see
	Remark~\ref{rem-sign-T}.

	By Lemma \ref{lem2}, $\sigma_{2k}(\hat\Omega)=0$ for all $k\ge1$.
	Since $\hat G$ is Berwald, $\hat\Omega$ is the curvature matrix of a
	linear connection on \(M\); the local spherical charts cover \(M\),
	so $\sigma_{2k}(\hat\Omega)\equiv0$ on \(M\).
By \eqref{4}, all
	rational Pontrjagin classes of \(M\) vanish. The argument actually
	shows that the Pontrjagin \emph{forms} $\sigma_{2k}(\hat\Omega)$
	vanish identically on \(M\), not merely their cohomology classes; see
	Remark~\ref{rem:stronger}.
\end{proof}

\begin{rem}\label{rem:stronger}
	The proof of Theorem~\ref{The 1} actually yields a stronger geometric
	statement than the cohomological vanishing asserted there. Fix a
	smooth positive density \(dv\) on \(M\) and let \(\hat G\) be the
	associated spray \eqref{1}. By Lemma~\ref{lem 1}, \(\hat G\) is
	Berwald whenever \(G\) is projectively affine; hence \(\hat G\)
	corresponds to a torsion-free linear connection \(\hat\Gamma\) on
	\(M\), whose curvature forms \(\hat\Omega^i_j\) are ordinary
	differential forms on \(M\). The argument in Section~\ref{p4.2} gives
	\[
	\sigma_{2k}(\hat\Omega)\equiv0\quad\text{on }M,\qquad k\ge1,
	\]
	as an identity of differential forms, not merely
	\([\sigma_{2k}(\hat\Omega)]=0\) in \(H^{4k}(M;\mathbb{R})\). Thus the
	curvature of \(\hat\Gamma\) is forced into an algebraic class whose
	even elementary symmetric polynomials vanish pointwise; the
	cohomological vanishing of Theorem~\ref{The 1} is an immediate
	consequence. This strengthened form does not depend on \(dv\); see
	Remark~\ref{rem-dv}.
\end{rem}

Since Berwald sprays and locally projectively flat sprays are both
special cases of projectively affine sprays, Theorem \ref{The 1}
directly yields the following two corollaries.

	\begin{cor}\label{cor-berwald-eta}
		Let \(M\) be an \(n\)-dimensional smooth manifold with \(n\ge2\). If
		\(M\) admits a locally spherically symmetric Berwald spray, then all
		rational Pontrjagin classes of \(M\) vanish.
	\end{cor}

	\begin{cor}\label{cor-s}
		Let \(M\) be an \(n\)-dimensional smooth manifold with \(n\ge2\).
		Suppose that \(M\) admits a locally spherically symmetric spray \(G\)
		which is locally projectively flat. Then all rational Pontrjagin
		classes of \(M\) vanish.
	\end{cor}

In the spherically symmetric setting, the latter corollary recovers
the result of Shen and Zhao \cite{ref1}.

	\section{\textbf{Examples}}\label{ex}		
In this section we construct spherically symmetric projectively affine
sprays that are not projectively flat, illustrating the scope of
Theorem \ref{The 1} beyond the locally projectively flat case of
\cite{ref1}; some of them are also non-Berwald, proving the existence
result Theorem \ref{thm-existence}.

	To verify that the examples below are not of scalar curvature, we use the following criterion from \cite{ref18}.
	\begin{lem}[\cite{ref18}]\label{lem3}
		Let \(G\) be a spherically symmetric spray on an open subset \(\Omega\subset \mathbb{R}^n\) with \(G^i\) expressed as in \eqref{GI}. Then \(G\) has scalar curvature if and only if $P(r,s)=0$, where $P(r,s)$ is given by \eqref{H1}.
	\end{lem}	

By Lemma \ref{g(r)=0}, a spherically symmetric spray is projectively
affine if and only if \(B(r,s)=f(r)s^2+h(r)\), with \(A(r,s)\)
arbitrary. Substituting into \eqref{H1} gives
\[
P(r,s)=4rf(r)h(r)-2f(r)+4h(r)^2+4h'(r),
\]
so \(P\not\equiv0\) implies that the spray is not of scalar curvature
(Lemma \ref{lem3}).

For the Berwald case, Theorem \ref{Thm 2} characterizes Berwald sprays
by \(A(r,s)=a(r)s\) and \(B(r,s)=f(r)s^2+h(r)\); by
Corollary \ref{cor-berwald-eta}, this class automatically satisfies the
hypothesis of Theorem \ref{The 1}.

\begin{ex}\label{ex:1}
	Take \(A(r,s)=0\), \(B(r,s)=c \neq0\), so that
	\(G^i=c|y|^2x^i\). By Theorem \ref{Thm 2}, \(G\) is a Berwald spray; moreover
	\(P(r,s)=4c^2\neq0\), so \(G\) is not of scalar curvature
	(Lemma \ref{lem3}). Hence Theorem \ref{The 1} applies to \(G\).
\end{ex}

We next construct non-Berwald examples with
\(A(r,s)=a(r)s+b(r)\) and \(B(r,s)=h(r)\), so that
\[
G^i=a(r)\langle x,y\rangle y^i+b(r)|y|y^i+h(r)|y|^2x^i.
\]
By Lemma \ref{g(r)=0} every spray in this family is projectively
affine; if \(b\not\equiv0\), it is not a Berwald spray by Theorem \ref{Thm 2}.
Moreover, \eqref{H1} reduces to
\begin{align}
	P(r,s)=4h'(r)+4h(r)^2,\label{h}
\end{align}
so \(P\not\equiv0\) implies that the spray is not of scalar curvature
(Lemma \ref{lem3}).

\begin{ex}\label{ex:2}
	Take \(A(r,s)=c_1e^{-c_2r}\) and \(B(r,s)=c_2\) with
	\(c_1,c_2\neq0\). Then
	\(G^i=c_1|y|e^{-c_2r}y^i+c_2|y|^2x^i\) is projectively affine, and
	since \(A\) is not of the form \(a(r)s\), \(G\) is not a Berwald spray.
	Equation \eqref{h} gives \(P(r,s)=4c_2^2\neq0\), so \(G\) is not of
	scalar curvature; since a projectively flat spray is necessarily of
	scalar curvature, \(G\) is not projectively flat.
\end{ex}

A variant with \(A=e^{-r}\) and \(B=s^2+1\) gives
\(G^i=|y|e^{-r}y^i+(\langle x,y\rangle^2+|y|^2)x^i\), which is again
projectively affine (Lemma \ref{g(r)=0}) and non-Berwaldian
(Theorem \ref{Thm 2}), with \(P(r,s)=4r+2>0\).

\begin{proof}[Proof of Theorem \ref{thm-existence}]
	Example \ref{ex:2} and the variant above are two spherically
	symmetric projectively affine sprays that are non-Berwald and not of
	scalar curvature. Neither is projectively flat, and by the proof of
	Theorem \ref{The 1} the Pontrjagin forms of their associated sprays
	vanish identically.
\end{proof}

For a spray \(G\), write \(\breve{\Omega}=(\breve{\Omega}^i_j)\) for the
curvature \(2\)-forms of its Berwald connection \eqref{Ber}, to
distinguish it from the curvature \(\hat{\Omega}\) of the associated
spray used in Section \ref{p4}. In dimension
four, Newton's identity (Lemma \ref{lem-newton}) gives
\begin{equation} \label{ex_sigma_2}
\sigma_2(\breve{\Omega})
=\frac{1}{2}\Bigl[\bigl(\operatorname{tr}\breve{\Omega}\bigr)^2
-\operatorname{tr}\bigl(\breve{\Omega}\wedge\breve{\Omega}\bigr)\Bigr],
\end{equation}
so to show \(\sigma_2(\breve\Omega)\not\equiv0\) it suffices to exhibit a
nonzero local \(4\)-form component.

We conclude this section by showing that the projectively affine
hypothesis in Theorem \ref{The 1} cannot be removed. On a
four-dimensional Euclidean domain \(U\subset\mathbb{R}^4\) we consider a
spherically symmetric spray that is not projectively affine, and whose
second Berwald Chern--Weil form does not vanish identically. This example
is local and is not intended as a converse to Theorem \ref{The 1}; it
illustrates the sharpness of the projectively affine condition.

\begin{ex}\label{ex:4}
	Consider \(G^i=|y|y^i+|y|\langle x,y\rangle x^i\), so that
	\(A(r,s)=1\), \(B(r,s)=s\). Since \(B\) is not of the form
	\(f(r)s^2+h(r)\), \(G\) is not projectively affine by
	Lemma \ref{g(r)=0}.

	We show directly that the second Berwald Chern--Weil form of \(G\)
	does not vanish. The Berwald connection and curvature forms are
	given by $\omega^i_j=\Gamma^i_{jk}\,dx^k$ and
	\begin{equation}\label{breveOmega}
		\breve{\Omega}^i_j = d\omega^i_j - \omega^m_j \wedge \omega^i_m =
		\frac{\partial \Gamma^i_{jk}}{\partial x^l}dx^l \wedge dx^k
		+ \frac{\partial \Gamma^i_{jk}}{\partial y^l}dy^l \wedge dx^k
		- \Gamma^m_{jk}\Gamma^i_{ml}\,dx^k \wedge dx^l.
	\end{equation}
	We work in dimension four and compute the coefficient of the
	\(4\)-form \(dx^1\wedge dx^2\wedge dy^1\wedge dy^2\). Since
	\(\breve{\Omega}\) contains no \(dy\wedge dy\) terms, the only
	components that can contribute to this coefficient in
	\(\breve{\Omega}\wedge\breve{\Omega}\) and
	\(\operatorname{tr}\breve{\Omega}\wedge\operatorname{tr}\breve{\Omega}\)
	are the \(dy^l\wedge dx^k\) terms with $k, l \in \{1, 2\}$.

	Using
	\( \partial|y| / \partial y^j= y_j /|y| \) and
	\( \partial\langle x,y\rangle / \partial y^j=x_j\), a direct
	computation gives
	\begin{align*}
		\Gamma^i_{jk}
		&=\frac{\delta_{jk}y^i+y_j\delta^i_k+y_k\delta^i_j}{|y|}
		-\frac{y_jy_ky^i}{|y|^3}
		+x^i\Bigl[\langle x,y\rangle\mathcal H_{jk}
		+\frac{y_jx_k+y_kx_j}{|y|}\Bigr],
	\end{align*}
	where $\mathcal H_{jk}=\delta_{jk}/|y|-y_jy_k/|y|^3$, and
	\begin{align}
		\Gamma^i_{jk\cdot l}:= \frac{\partial\Gamma^i_{jk}}{\partial y^l}
		&=\delta^i_l\mathcal H_{jk}+y^i\frac{\partial\mathcal H_{jk}}{\partial y^l}
		+\frac{\delta_{jl}\delta^i_k+\delta_{kl}\delta^i_j}{|y|}
		-\frac{(y_j\delta^i_k+y_k\delta^i_j)y_l}{|y|^3} \notag \\
		&\quad+x^i\Bigl[x_l\mathcal H_{jk}
		+\langle x,y\rangle\frac{\partial\mathcal H_{jk}}{\partial y^l}
		+\frac{\delta_{jl}x_k+\delta_{kl}x_j}{|y|}
		-\frac{(y_jx_k+y_kx_j)y_l}{|y|^3}\Bigr], \label{ex_Gamma_dot}
	\end{align}
	with $\partial\mathcal H_{jk}/\partial y^l
	=-(\delta_{jk}y_l+\delta_{jl}y_k+\delta_{kl}y_j)/|y|^3
	+3y_jy_ky_l/|y|^5$.

	Set \(x_0=(1,0,0,0)\) and
	\(y_0=(0,\frac{\sqrt{2}}{2},\frac{\sqrt{2}}{2},0)\); then
	\(|y_0|=1\) and \(\langle x_0,y_0\rangle=0\). At this point,
	\eqref{ex_Gamma_dot} reduces to
	\begin{align}\label{Gamma_ijkl_0}
		\Gamma^i_{jk\cdot l}& =
		\delta_{jk}\delta^i_l + \delta_{jl}\delta^i_k+\delta_{kl}\delta^i_j - \delta^i_l y_j y_k
		- \delta^i_k y_j y_l - \delta^i_j y_k y_l
		 + y^i \Bigl[ 3y_j y_k y_l- \delta_{jk} y_l-\delta_{jl}y_k-\delta_{kl} y_j \Bigr] \notag\\
		&\quad + x^i\Bigl[ \delta_{jk}x_l + \delta_{jl}x_k+ \delta_{kl}x_j  -y_j y_k x_l
		- y_j y_l x_k - y_k y_l x_j  \Bigr].
	\end{align}
	A direct substitution into \eqref{Gamma_ijkl_0} gives the
	non-vanishing components
	\[
		\begin{aligned}
			&\Gamma^1_{11\cdot 1} = 6, \quad \Gamma^1_{12\cdot 2} = 1, \quad
			\Gamma^1_{21\cdot 2} = 1, \quad \Gamma^1_{22\cdot 1} = 1, \quad
			\Gamma^1_{31\cdot 2} = -1, \quad \Gamma^1_{32\cdot 1} = -1, \\
			&\Gamma^2_{11\cdot 2} = \tfrac{1}{2}, \quad \Gamma^2_{12\cdot 1} = \tfrac{1}{2}, \quad
			\Gamma^2_{21\cdot 1} = \tfrac{1}{2}, \quad \Gamma^2_{22\cdot 2} = \tfrac{3}{4}, \quad
			\Gamma^2_{31\cdot 1} = -\tfrac{1}{2}, \quad \Gamma^2_{32\cdot 2} = -\tfrac{3}{4}, \\
			&\Gamma^3_{11\cdot 2} = -\tfrac{1}{2}, \quad \Gamma^3_{12\cdot 1} = -\tfrac{1}{2}, \quad
			\Gamma^3_{21\cdot 1} = -\tfrac{1}{2}, \quad \Gamma^3_{22\cdot 2} = -\tfrac{3}{4}, \quad
			\Gamma^3_{31\cdot 1} = \tfrac{1}{2}, \quad \Gamma^3_{32\cdot 2} = \tfrac{3}{4}, \\
			&\Gamma^4_{41\cdot 1} = 1, \quad \Gamma^4_{42\cdot 2} = \tfrac{1}{2},
		\end{aligned}
	\]
	and all other components with $k,l \in \{1, 2\}$ vanish at \((x_0,y_0)\).

	Here and below, \(\dot{\Omega}^i_j\) denotes terms that do not
	contribute to the coefficient of
	\(dx^1\wedge dx^2\wedge dy^1\wedge dy^2\). Substituting the above
	values into \eqref{breveOmega} yields
	\begin{align*}
		\breve{\Omega}^1_1 &= 6\, dy^1 \wedge dx^1 + dy^2 \wedge dx^2 + \dot{\Omega}^1_1, \quad
		\breve{\Omega}^1_2 = dy^1 \wedge dx^2 + dy^2 \wedge dx^1 + \dot{\Omega}^1_2,  \\
\breve{\Omega}^1_3 &= -dy^1 \wedge dx^2 - dy^2 \wedge dx^1 + \dot{\Omega}^1_3,
\end{align*}
\begin{align*}
		\breve{\Omega}^2_1 &= \tfrac{1}{2}\, dy^1 \wedge dx^2 + \tfrac{1}{2}\, dy^2 \wedge dx^1 + \dot{\Omega}^2_1, \quad
		\breve{\Omega}^2_2 = \tfrac{1}{2}\, dy^1 \wedge dx^1 + \tfrac{3}{4}\,dy^2 \wedge dx^2 + \dot{\Omega}^2_2, \\
\breve{\Omega}^2_3 &= -\tfrac{1}{2}\, dy^1 \wedge dx^1 - \tfrac{3}{4}\,dy^2 \wedge dx^2 + \dot{\Omega}^2_3,  \\
		\breve{\Omega}^3_1 &= -\tfrac{1}{2}\, dy^1 \wedge dx^2 - \tfrac{1}{2}\, dy^2 \wedge dx^1 + \dot{\Omega}^3_1, \quad
		\breve{\Omega}^3_2 = -\tfrac{1}{2}\, dy^1 \wedge dx^1 - \tfrac{3}{4}\,dy^2 \wedge dx^2 + \dot{\Omega}^3_2, \\
		\breve{\Omega}^3_3 &= \tfrac{1}{2}\, dy^1 \wedge dx^1 + \tfrac{3}{4}\,dy^2 \wedge dx^2 + \dot{\Omega}^3_3, \quad
		\breve{\Omega}^4_4 = dy^1 \wedge dx^1 +\tfrac{1}{2}\, dy^2 \wedge dx^2 + \dot{\Omega}^4_4,
	\end{align*}
	with $\breve{\Omega}^1_4, \breve{\Omega}^2_4, \breve{\Omega}^3_4,
	\breve{\Omega}^4_1, \breve{\Omega}^4_2, \breve{\Omega}^4_3$ entirely
	absorbed into their respective $\dot{\Omega}^i_j$.

	A direct computation of the diagonal components of
	\(\breve{\Omega}\wedge\breve{\Omega}\) yields
	\begin{align*}
		\bigl( \breve{\Omega}\wedge\breve{\Omega} \bigr)^1_1 &= -10 \, dx^1\wedge dx^2\wedge dy^1\wedge dy^2 + \dot{W}^1_1, \quad
		\bigl( \breve{\Omega}\wedge\breve{\Omega} \bigr)^2_2 = -\frac{1}{2} \, dx^1\wedge dx^2\wedge dy^1\wedge dy^2 + \dot{W}^2_2, \\
		\bigl( \breve{\Omega}\wedge\breve{\Omega} \bigr)^3_3 &= -\frac{1}{2} \, dx^1\wedge dx^2\wedge dy^1\wedge dy^2 + \dot{W}^3_3, \quad
		\bigl( \breve{\Omega}\wedge\breve{\Omega} \bigr)^4_4 = -\, dx^1\wedge dx^2\wedge dy^1\wedge dy^2 + \dot{W}^4_4,
	\end{align*}
	where \(\dot{W}^i_i\) collects the remaining cross terms. Thus,
	\[
		\operatorname{tr} \bigl( \breve{\Omega}\wedge\breve{\Omega} \bigr)
		= -12 \, dx^1\wedge dx^2\wedge dy^1\wedge dy^2 + \dot{W}.
	\]
	Taking the trace of \(\breve{\Omega}\) and retaining the explicit
	components together with the omitted terms, we get
	\[
		\operatorname{tr}\breve{\Omega}
		= 8\,dy^1\wedge dx^1+ 3 \,dy^2\wedge dx^2 + \dot{T},
	\]
	where \(\dot{T}\) collects the remaining components. Therefore,
	\begin{align*}
		\operatorname{tr}\breve{\Omega}\wedge\operatorname{tr}\breve{\Omega}
		&= -48\,dx^1\wedge dx^2\wedge dy^1\wedge dy^2
		+ 2\bigl(8\,dy^1\wedge dx^1+3\,dy^2\wedge dx^2\bigr)\wedge\dot{T}
		+ \dot{T}\wedge\dot{T}.
	\end{align*}

	By Newton's identity, the coefficient of
	\(dx^1\wedge dx^2\wedge dy^1\wedge dy^2\) in
	\(\sigma_2(\breve{\Omega})\) equals
	\[
		\frac{1}{2} \bigl[-48-(-12)\bigr] = -18.
	\]
	Hence \(\sigma_2(\breve{\Omega})\not\equiv0\). This shows that
	spherical symmetry alone does not force the second Berwald
	Chern--Weil form to vanish identically.
\end{ex}

Since \(U\) is contractible, every closed form on \(U\) is exact, so all
Pontrjagin classes of \(U\) vanish trivially. The point of Example
\ref{ex:4} is different: it shows that the Berwald Chern--Weil form
\(\sigma_2(\breve\Omega)\) is not identically zero as a differential
form.
Since the slit tangent bundle over a contractible domain is contractible,
this nonvanishing form is exact there; hence there is no contradiction with the vanishing of the cohomology class.
The example therefore shows that
spherical symmetry alone does not force \(\sigma_2(\breve\Omega)\) to
vanish, and hence that the projectively affine condition in Theorem
\ref{The 1} is essential for the vanishing of the \emph{form}.

	\phantomsection

\end{document}